\documentclass[a4paper,10pt]{article}
\usepackage{a4}
\usepackage{amsmath,amsthm,amssymb,enumerate}
\usepackage{amsthm}

\usepackage{amsfonts, amsmath}   
\usepackage{amssymb}
\usepackage{amsmath}
\usepackage{graphics}

\usepackage{amsmath,amsthm,amssymb,enumerate}
\usepackage{amsthm}
\newtheorem{theorem}{Theorem}[section]

\newtheorem{prop}[theorem]{Proposition}

\newtheorem{dfn}[theorem]{Definition}

\newtheorem{con}[theorem]{Conjecture}

\numberwithin{equation}{section}

\begin{document}
\title{\textbf{ Majority out-dominating functions in digraphs} }
\author{ {\bf \sc   Mart\'{\i}n Manrique$^*$,   Karam Ebadi$^\dag$, and Akbar Azami$^\dag$  }\\ \\{\footnotesize $^*$ $n$-CARDMATH, Kalasalingam University, India.} \\ {\footnotesize  {\it   martin.manrique@gmail.com}}\\\\ {\footnotesize  $^\dag$ Department of Mathematics, Islamic Azad University of Miandoab, Iran.}\\{\footnotesize  {\it karam\_ebadi@yahoo.com, akbarazami10@yahoo.com }}\\
}

\date{}
\maketitle

\begin{abstract}
At least two different notions have been published under the name "majority domination in graphs": Majority dominating functions and majority dominating sets. In this work we
extend the former concept to digraphs.  Given a digraph $D=(V,A),$ a function $f : V \rightarrow \{-1,1\}$ such that $f(N^+[v])\geq1$ for at least half of the vertices $v$ in $V$ is  a \textit{ majority out-dominating function } (MODF)  of $D.$ The \textit{weight} of a MODF $f$ is $w(f)=\sum\limits_{v\in V}f(v),$ and the minimum weight of a  MODF in $D$ is
the {\it majority out-domination number}   of $D,$ denoted $\gamma^+_{maj}(D).$ In this work we introduce these concepts and prove some results regarding them, among which the fact that the decision problem of finding a  majority out-dominating function of a given weight is NP-complete.
\end{abstract}



\section{Introduction}

Several different kinds of situations can be modeled with the help of majority out-dominating functions in digraphs. Basically, by assigning either $-1$ or $+1$ to each vertex of a (di)graph, we are partitioning the set of vertices into a set of "bad" elements and a set of "good" elements. In this context, adjacency may be interpreted as "influence". In a graph, if element $u$ influences element $v,$ then $v$ influences $u$ as well; in a digraph this does not necessarily hold.

Typical examples of applications of majority dominating functions, both in digraphs and in undirected graphs, arise in the context of democracy. However, there are other areas where they are useful.

For example, consider the case of a natural reserve. We can assign $-1$ to communities inside the reserve where some furtive hunters live, and $1$ to ranger posts. Adjacency means the possibility to reach one place from another in a short time. It is directional due to the characteristics of the land (for example, up-hill vs. down-hill). If $f(N^+[v])\geq1$ for a given spot $v$, this means there are enough rangers to stop the hunting there. Then a MODF of minimum weight is a distribution with a minimum number of ranger posts such that at least half of the reserve is safe. Of course, it would be better to protect the whole reserve; this kind of application becomes relevant when the number of rangers available is very limited.\\

Section 2 is focused on the basic definitions used in the paper. In Section 3 we determine $\gamma^+_{maj}$ for some standard classes of digraphs and prove general results about MODFs, mainly regarding the deletion of a vertex and that of an arc, as well as the inversion of an arc. In Section 4 we deal with orientations of a graph $G$: Which is the maximum and minimum for $\gamma^+_{maj}(D)$ such that $D$ is an orientation of $G$? We answer the question for some standard classes of graphs. Section 5 is the proof that the decision problem of finding a MODF of a given weight is NP-complete.

\section{Fundamentals}

Throughout this paper $D=(V,A)$ is a finite directed graph with neither
loops nor multiple arcs (but pairs of opposite arcs are allowed) and $%
G=(V,E) $ is a finite undirected graph with neither loops nor multiple
edges. Unless stated otherwise, $n$ denotes the order of $D$ (or $G$), that is, $n=|V|$. For basic terminology on graphs and digraphs we refer to \cite{cl}, and for a monograph regarding domination in digraphs we refer to Chapter 15 in \cite{h}.

Let $G=(V,E)$ be a graph. For any vertex $u\in V$, the set $%
N_G(u)=\{v: uv\in E\}$ is called the
neighborhood  of $u$ in $G.$  $N_G[u]=N_G(u)\cup \{u\}$ is the closed neighborhood of $u$ in $G.$ The degree of $u$ in $G$ is $%
d_G(u)=|N_G(u)|.$ When the graph $G$ is clear from the context, we may write simply $N(u),$  $N[u],$ and $d(u).$

Let $D=(V,A)$ be a digraph. For any vertex $u\in V$, the sets $%
N^{+}_D(u)=\{v: uv\in A\}$ and $N^{-}_D(u)=\{v: vu\in A\}$ are called the
out-neighborhood and in-neighborhood of $u$ in $D,$ respectively. $N^{+}_D[u]=N^{+}_D(u)\cup \{u\}$ is the closed out-neighborhood of $u$ in $D,$ and $N^{-}_D[u]=N^{-}_D(u)\cup \{u\}$ is the closed in-neighborhood of $u$ in $D.$ The in-degree and out-degree of $u$ in $D$ are defined by $%
d^{-}_D(u)=|N^{-}_D(u)|$ and $d^{+}_D(u)=|N^{+}_D(u)|,$ respectively. The maximum out-degree of the vertices of $D$ is $\delta^+ (D)$. A set $S\subseteq V$ is in-dominating (or absorbent) if every vertex $v\in V\setminus S$ has an out-neighbor in $S$. The minimum cardinality of an in-dominating set in $D$ is denoted $\gamma^- (D)$.  When the digraph $D$ is clear from the context, we may write simply $N^+(u),$  $d^+(u),$ etc.\\

Let $G=(V,E)$ be a graph. A majority dominating function \cite{mf} is a function
$f:V\rightarrow \{-1,1\}$ such that the set $S=\{v\in V:\sum\limits_{u\in
N[v]}f(u)\geq 1\}~$satisfies $|S|~\geq \frac{n}{2};$ the weight of a
majority dominating function is $w(f)=\sum\limits_{v\in V}f(v),$ and $\min
\{w(f):f$ is a majority dominating function in $G\}$ is the majority
domination number of $G,$ denoted $\gamma _{maj}(G);$ a majority dominating
function $f$ in $G$ such that $w(f)=\gamma _{maj}(G)$ is a $\gamma _{maj}(G)$%
-function.

This concept can be naturally extended to digraphs: Given a digraph $D=(V,A),$ a
majority out-dominating function (MODF) of $D$ is a function $f:V\rightarrow
\{-1,1\}$ such that the set $S=\{v\in V:\sum\limits_{u\in N^{+}[v]}f(u)\geq
1\}~$satisfies $|S|~\geq \frac{n}{2};$ the weight of $f$ is $%
w(f)=\sum\limits_{v\in V}f(v),$ and $\min \{w(f):f$ is a MODF of $D\}$ is the majority out-domination number of $D,
$ denoted $\gamma _{maj}^{+}(D);$ a MODF $f$ of $%
D$ such that $w(f)=\gamma _{maj}^{+}(D)$ is a $\gamma _{maj}^{+}(D)$%
-function. For a set $X\subseteq V$ and a function $f:V\rightarrow
\{-1,1\},$ we will use $f(X)=\sum\limits_{v\in X}f(v),$ so $w(f)=f(V).$

Of course, we can define majority in-dominating functions analogously. However, as happens with kernels and solutions, the concepts are equivalent since a function $f$ is a majority in-dominating function of a given digraph $D$ if, and only if, $f$ is a MODF of $\overleftarrow{D},$ where $\overleftarrow{D}$ is the digraph resulting from reversing all the arcs of $D.$

It is also possible to define majority out-dominating sets in digraphs. This is done in \cite{mase}. However, the relation between this notion and that of majority out-dominating function is weak, and was pointed out in the paper mentioned, so we will not even define the concept here.

\section{MODFs in digraphs}

We start this section by determining the majority out domination number for some special classes of digraphs. Observe that for a digraph $D=(V,A)$ and a MODF $f$ of $D,$ $w(f)=|f^{-1}(1)|-|f^{-1}(-1)|$:

\begin{prop}\label{prop34}
Let $C_n$ denote the directed cycle with $n$ vertices. Then
 \[\gamma^+_{maj}(C_n) = \left\lbrace
  \begin{array}{c l}
  2  & \text{if  $  n$ is even }\\
    3  & \text{if  $  n$ is odd }
  \end{array}
\right. \]
\end{prop}

\begin{proof}
We number the vertices of $C_n=(V,A)$ in order, that is, $V=\{v_1,...,v_n\}$ with $N^+(v_i)=\{v_{i+1}\},$ where "$+$" denotes the addition modulo $n.$  Let $f$ be a MODF of $C_n.$ Notice that if $f(v_j)=-1,$ then $f(N^+[v_j])$ and $f(N^+[v_{j-1}])$ are non-positive, where "$-$" denotes the inverse operation of $+.$ Therefore, if we have $f(x)=-1$ for $\lfloor\frac{n}2\rfloor$ vertices, then $f(N^+[x])$ will be non-positive for at least $\lfloor\frac{n}2\rfloor+1> \frac{n}2 $ vertices. It follows that at least $\lceil \frac{n}2 \rceil+1$ vertices must satisfy  $f(x)=1.$ Therefore, $\gamma^+_{maj}(C_n)\geq 2$ if $n$ is even, and $\gamma^+_{maj}(C_n)\geq 3$ if $n$ is odd.

On the other hand, the function $g:V\rightarrow\{-1,1\}$ such that $g(v_i)=1$ for $1\leq i\leq\lceil \frac{n}2 \rceil+1$ and $g(v_i)=-1$ otherwise,  is a MODF of $C_n.$
\end{proof}

\begin{prop} \label{prop333}
Let $P_n$ denote the directed path with $n$ vertices. Then
 \[\gamma^+_{maj}(P_n) = \left\lbrace
  \begin{array}{c l}
  0  & \text{if  $  n$ is even }\\
    1  & \text{if  $  n$ is odd }
  \end{array}
\right. \]
\end{prop}

\begin{proof}
The proof is similar to that of Proposition \ref{prop34}: We number the vertices of $P_n=(V,A)$ in order, that is, $V=\{v_1,...,v_n\}$ with $N^+(v_i)=\{v_{i+1}\},$ for $i\in\{1,...,n-1\}.$  Then for every   MODF $f$ of $P_n$ and every $j\in\{2,...,n\},$  if $f(v_j)=-1$ we have that $f(N^+[v_j])$ and $f(N^+[v_{j-1}])$ are non-positive, while $f(v_1)=-1$ implies as well that $f(N^+[v_1])$ is non-positive.   Therefore,  at least $\lfloor \frac{n-1}2 \rfloor+1=\lceil \frac{n}2 \rceil$ vertices must satisfy  $f(x)=1.$ It follows that $\gamma^+_{maj}(P_n)\geq 0$ if $n$ is even, and $\gamma^+_{maj}(P_n)\geq 1$ if $n$ is odd.

Conversely, the function $g:V\rightarrow\{-1,1\}$ such that $g(v_i)=-1$ for $1\leq i\leq\lfloor\frac{n}2\rfloor$ and $g(v_i)=1$ for  $\lfloor\frac{n}2\rfloor+1\leq i\leq n$ is a MODF of $P_n.$
\end{proof}

\begin{dfn}
A digraph $D=(V,A)$ is transitive if for every $\{u,v,w\}\subseteq V$, $uv\in A$ and $vw\in A$ imply $uw\in A$. A tournament is an orientation of a complete graph.
\end{dfn}

\begin{theorem} \label {thm 35}
For every transitive tournament $T$, $\gamma^+_{maj}(T)=-n+2\lceil \frac{n+2}4 \rceil.$
\end{theorem}

\begin{proof}
Let $T=(V,A)$ be a transitive tournament of order $n.$ Notice that both the in-degree sequence and the out-degree sequence of $T$ are $(n-1, n-2, ..., 1, 0).$  Let $f$ be a MODF  of $T,$ and suppose  that $w(f)<-n+2\lceil \frac{n+2}4 \rceil.$ Then $|f^{-1}(1)|\leq\lceil \frac{n+2}4 \rceil-1.$ Therefore, even if the vertices whose value under $f$ is $1$ are those having the greatest in-degree, at most $\lceil \frac{n+2}4 \rceil-2$ vertices $x$ satisfy $f(x)=-1$ and $f(N^+[x])\geq1.$ Since $2\lceil \frac{n+2}4 \rceil-3< \frac{n}2,$ it follows that $f$ is not a MODF of $T,$ which is a contradiction. Then  $|f^{-1}(1)|\geq \lceil \frac{n+2}4 \rceil,$ that is, $w(f)\geq -n+2\lceil \frac{n+2}4 \rceil.$  Therefore, $\gamma^+_{maj}(T)\geq -n+2\lceil \frac{n+2}4 \rceil.$

On the other hand, the function $g:V\rightarrow\{-1,1\}$ such that $g(v)=1$ if $d^+(v)< \lceil \frac{n+2}4 \rceil$ and $g(v)=-1$ if $d^+(v)\geq \lceil \frac{n+2}4 \rceil$ is a MODF of $T,$ because $2\lceil \frac{n+2}4 \rceil-1\geq \frac{n}2.$
\end{proof}

In relation with Theorem \ref{thm 35}, we have the following conjecture:

\begin{con} \label{con 36}
Let $D=(V,A)$ be a digraph such that for every $v\in V, \ d^-(v)+d^+(v)=c$ for some natural number $c$ (that is, the underlying undirected graph of $D$ is regular). Then there is a $\gamma^+_{maj}(D)$-function $f$ such that for every $\{u,v\}\subseteq V,$ $(f(u)=-1, \ f(v)=1) \Rightarrow d^-(u)\leq d^-(v).$
\end{con}

Intuitively, if we assign $-1$ to the vertices with least in-degree, less vertices will be affected by them, so more vertices may have the value $-1$ while the function is still a MODF. It may be easier to prove the conjecture for tournaments.\\

Now we will prove some general results regarding MODFs:

\begin{dfn}
Let $D=(V,A)$ be a digraph, and let $f$ be a MODF of $D.$ Then $f$ is minimal if there is no $g\neq f$ such that $g$ is a MODF of $D$ and $g(v)\leq f(v)$ for every $v\in V.$
\end{dfn}

\begin{prop} \label{prop44}
Let $D=(V,A)$ be a digraph, and let $f$ be a minimal MODF of $D.$ Then for every $v\in V$ with $f(v)=1,$ there exists $u\in N^-[v]$ such that $f(N^+[u])\in \{1,2\}.$
\end{prop}

\begin{proof}
Let $f$ be a MODF of the digraph $D=(V,A).$ Take $v\in V$ with $f(v)=1,$ and suppose that for every $u\in N^-[v]$ we have $f(N^+[u])\notin \{1,2\}.$ Define the function $g:V\rightarrow\{-1,1\}$ as follows: $g(v)=-1,$ and $g(x)=f(x)$ for every $x\in V\setminus\{v\}.$ Then for every $u\in N^-[v]$ we have $g(N^+[u])=f(N^+[u])-2,$ and for every $x\in V\setminus N^-[v]$ we have $g(N^+[u])=f(N^+[u]).$ Therefore, for every vertex $y\in V,$ $g(N^+[y])$ is positive if, and only if, $f(N^+[y])$ is positive, so $g$ is a MODF of  $D.$ Since for every vertex $y\in V,$ $g(y)\leq f(y),$ it follows that $f$ is not minimal.
\end{proof}

It is clear that the converse of Proposition \ref{prop44} does not hold. For example, in a directed path $P_n$ consider the function assigning $1$ to every vertex.\\

It is interesting to explore the effect in $\gamma^+_{maj}$ of the removal of an arc or a vertex, as well as that of reversing one arc:

\begin{theorem}\label{prop 31}
Let $D=(V,A)$ be a digraph, and let $D'=(V,A')$ be a digraph obtained by reversing one arc of $D,$ that is, for some $uv\in A,$ $A'=A\setminus\{uv\}\cup\{vu\}.$ Then $\gamma^+_{maj}(D)-2\leq \gamma^+_{maj}(D')\leq\gamma^+_{maj}(D)+2.$ The bounds are sharp.
\end{theorem}

\begin{proof}
Let $f$ be a $\gamma^+_{maj}(D)$-function and take $uv\in A.$ If $f(u)=f(v)=-1,$ then $f(N^+_{D'}[u])= f(N^+_{D}[u])+1,$ and $f(N^+_{D'}[v])= f(N^+_{D}[v])-1,$ while $f(N^+_{D'}[x])= f(N^+_{D}[x])$ for every $x\in V\setminus \{u,v\}.$ Therefore, the function $g:V\rightarrow \{-1,1\}$ such that $g(u)=1$ and $g(x)=f(x)$ for every $x\in V\setminus \{u\}$ is a MODF of $D'$ with $w(g)=w(f)+2.$

If $f(u)=f(v)=1,$ then $f(N^+_{D'}[u])= f(N^+_{D}[u])-1,$ and $f(N^+_{D'}[v])= f(N^+_{D}[v])+1,$ while $f(N^+_{D'}[x])= f(N^+_{D}[x])$ for every $x\in V\setminus \{u,v\}.$ If $f(N^+_{D}[u])\neq1,$ then $f$ is a MODF of $D'.$ If $f(N^+_{D}[u])=1,$ there exists $z\in N^+_{D}[u]\setminus \{v\}$ with $f(z)=-1.$ Then  the function $g:V\rightarrow \{-1,1\}$ such that $g(z)=1$ and $g(x)=f(x)$ for every $x\in V\setminus \{z\}$ is a MODF of $D'$ with $w(g)=w(f)+2.$

If $f(u)=1$ and $f(v)=-1,$ then $f(N^+_{D'}[u])= f(N^+_{D}[u])+1,$ and $f(N^+_{D'}[v])= f(N^+_{D}[v])+1,$ while $f(N^+_{D'}[x])= f(N^+_{D}[x])$ for every $x\in V\setminus \{u,v\}.$ Therefore,  $f$ is a MODF of $D'.$

If $f(u)=-1$ and $f(v)=+1,$ then $f(N^+_{D'}[u])= f(N^+_{D}[u])-1,$ and $f(N^+_{D'}[v])= f(N^+_{D}[v])-1,$ while $f(N^+_{D'}[x])= f(N^+_{D}[x])$ for every $x\in V\setminus \{u,v\}.$ However, if we consider  the function $g:V\rightarrow \{-1,1\}$ such that $g(u)=1$ and $g(x)=f(x)$ for every $x\in V\setminus \{u\},$ we have that $g(N^+_{D'}[u])= f(N^+_{D}[u])+1,$ and $g(N^+_{D'}[v])= f(N^+_{D}[v])+1,$ while $g(N^+_{D'}[x])\geq f(N^+_{D}[x])$ for every $x\in V\setminus \{u,v\}.$ Then $g$ is a MODF of $D'$ with $w(g)=w(f)+2.$ This settles the upper bound.

The lower bound follows because $D$ s obtained from $D'$ by reversing the arc $vu$. Hence $\gamma^+_{maj}(D)-2\leq \gamma^+_{maj}(D')\leq\gamma^+_{maj}(D)+2.$

Now, if $C_3$ denotes a directed triangle and $D$ is obtained from $C_3$ by reversing one arc, we have that $\gamma^+_{maj}(C_3)=3$ and $\gamma^+_{maj}(D)=1.$ Therefore, the bounds are sharp.
\end{proof}

\begin{theorem}
Let $D=(V,A)$ be a digraph with $uv\in A.$ Then $\gamma^+_{maj}(D)-2\leq \gamma^+_{maj}(D-uv)\leq\gamma^+_{maj}(D)+2.$ The bounds are sharp.
\end{theorem}

\begin{proof}
The upper bound  follows in a similar way to that of Proposition \ref{prop 31}. The same function works in each case (although in the first case $g$ is not needed, since $f$ is already a MODF of $D'$). To show that the bound is sharp, consider the orientation $D=(V,A)$ of the star $K_{1,4}$ such that $d^-(u)=d^+(u)=2$ for the central vertex $u.$ It is easy to verify that $\gamma^+_{maj}(D)=-1.$ For any $v\in N^+(u),$ the digraph $D-uv$ is isomorphic to the orientation of the star $K_{1,3}$ such that $d^-(u)= 2$ and $d^+(u)=1,$ plus an isolated vertex. Notice that for any function $g:V\rightarrow \{-1,1\},$ if $g(u)=-1$ then $g(N^+[x])\leq 0$ for every $x\in N^-[u],$ which implies that $g$ is not a MODF of $D-uv.$ Take a $\gamma^+_{maj}(D-uv)$-function $f$. Since $f(u)=1,$ it follows that for every $x\in V\setminus\{u\},$ $f(N^+[x])>0$ if, and only if, $f(x)=1.$ Therefore, $\gamma^+_{maj}(D-uv)=1.$

For the lower bound, assume $\gamma^+_{maj}(D-uv)<\gamma^+_{maj}(D)$ and take a $\gamma^+_{maj}(D-uv)$-function $g.$ Notice that $g(N^+_{D-uv}[x])=g(N^+_{D}[x])$ for every $x\in V\setminus \{u\}.$ Then $g(N^+_{D-uv}[u])<g(N^+_{D}[u]),$ so $g(v)=-1.$ Since the function $f:V\rightarrow \{-1,1\}$ such that $f(v)=1$ and $f(x)=g(x)$ for every $x\in V\setminus \{u\}$ is a MODF of $D,$ it follows that $\gamma^+_{maj}(D-uv)=w(g)= w(f)-2 \geq \gamma^+_{maj}(D)-2.$  As in the previous proposition, if $C_3$ denotes a directed triangle and $D$ is obtained from $C_3$ by deleting one arc, we have that $\gamma^+_{maj}(C_3)=3$ and $\gamma^+_{maj}(D)=1,$ so the bound is sharp.
\end{proof}

\begin{prop} \label{prop35}
Let $D=(V,A)$ be a digraph and take $v\in V$ with $d^{+}(v)=0$. Then $\gamma^+_{maj}(D)-1\leq \gamma^+_{maj}(D-v).$
\end{prop}

\begin{proof}
Take a  $\gamma^+_{maj}(D-v)$-function $g.$ Since the function $g':V\rightarrow \{-1,1\}$ such that $g'(v)=1$ and $g'(x)=g(x)$ for every $x\in V\setminus \{v\}$ is a MODF of $D,$ we have $w(g)=w(g')-1\geq \gamma^+_{maj}(D)-1.$
\end{proof}

In relation with Proposition \ref{prop35}, notice that the result does not necessarily hold if $d^{+}(v)>0$. For example, for the digraph $D=(V,A)$ shown in Figure 1 we have $\gamma^+_{maj}(D)=1$, while $\gamma^+_{maj}(D-v)=-2$. Following this idea it is easy to construct examples in which the difference is any positive integer.

\begin{center}
\unitlength 1mm 
\linethickness{0.4pt}
\ifx\plotpoint\undefined\newsavebox{\plotpoint}\fi 
\begin{picture}(75.75,23.75)(0,0)
\put(2.25,2.75){\circle*{2}}
\put(1.5,21.25){\circle*{2}}
\put(19.25,10.25){\circle*{2}}
\put(36.25,10.75){\circle*{2}}
\put(54.75,10.75){\circle*{2}}
\put(74.75,.75){\circle*{2}}
\put(74.75,22.75){\circle*{2}}
\put(64.75,16.875){\vector(3,2){.07}}\multiput(54.75,11)(.0573065903,.0336676218){349}{\line(1,0){.0573065903}}
\put(64.375,6.125){\vector(-2,1){.07}}\multiput(74.25,1)(-.0649671053,.0337171053){304}{\line(-1,0){.0649671053}}
\put(75,12.25){\vector(0,-1){.07}}\put(75,22.25){\line(0,-1){20}}
\put(45.5,10.75){\vector(1,0){.07}}\put(36.75,10.75){\line(1,0){17.5}}
\put(27.375,10.5){\vector(-1,0){.07}}\put(35.75,10.5){\line(-1,0){16.75}}
\put(9.875,16.25){\vector(-3,2){.07}}\multiput(18.5,11)(-.0552884615,.0336538462){312}{\line(-1,0){.0552884615}}
\put(10.5,6.5){\vector(-3,-1){.07}}\multiput(19.25,10)(-.084134615,-.033653846){208}{\line(-1,0){.084134615}}
\put(36,7.5){\makebox(0,0)[cc]{$v$}}
\end{picture}

Figure 1
\end{center}

In a similar way, the removal of a vertex may increase $\gamma^+_{maj}$ as much as desired. For example, consider the digraph $D=(V,A)$ shown in Figure 2, where $V=\{u,v\}\cup S\cup T,$ $|S|=k,$ $|T|=k+2,$ $d^-(x)=0$ for every $x\in S\cup T,$ $d^+(u)=d^+(v)=0,$ $N^-(u)=S,$ and $N^-(v)=S\cup T.$ Then the function $f:V\rightarrow \{-1,1\}$ such that $f(u)=f(v)=1$ and $f(x)=-1$ for every $x\in S\cup T$ is a MODF of $D$ with $w(f)= -2k.$ However, $D-u$ is a star in which the central vertex $v$ is the head of every arc, so $\gamma^+_{maj}(D-u)\geq0.$

\begin{center}
\unitlength 1mm 
\linethickness{0.4pt}
\ifx\plotpoint\undefined\newsavebox{\plotpoint}\fi 
\begin{picture}(68.75,45.25)(0,0)
\put(6.25,24.25){\circle*{2}}
\put(46,23.75){\circle*{2}}
\put(66.25,4){\circle*{2}}
\put(25.75,43.25){\circle*{2}}
\put(66.25,13.75){\circle*{2}}
\put(66.25,35.75){\circle*{2}}
\put(26.25,34){\circle*{2}}
\put(66.25,43){\circle*{2}}
\put(26,29.25){\circle*{1}}
\put(26,26){\circle*{1}}
\put(26,22.75){\circle*{1}}
\put(26,18.75){\circle*{1}}
\put(66,31.25){\circle*{1}}
\put(66,27){\circle*{1.12}}
\put(66,22.75){\circle*{1}}
\put(66,18.5){\circle*{1}}
\put(25.75,13.75){\circle*{2}}
\put(15.875,33.625){\vector(-1,-1){.07}}\multiput(25.5,43.25)(-.0337127846,-.0337127846){571}{\line(0,-1){.0337127846}}
\put(16.125,29){\vector(-2,-1){.07}}\multiput(26,34)(-.0664983165,-.0336700337){297}{\line(-1,0){.0664983165}}
\put(15.625,19){\vector(-2,1){.07}}\multiput(25.25,13.75)(-.0616987179,.0336538462){312}{\line(-1,0){.0616987179}}
\put(35.875,33.375){\vector(1,-1){.07}}\multiput(25.75,43.25)(.034556314,-.0337030717){586}{\line(1,0){.034556314}}
\put(36,28.75){\vector(2,-1){.07}}\multiput(26.25,34)(.0625,-.0336538462){312}{\line(1,0){.0625}}
\put(35.625,18.625){\vector(2,1){.07}}\multiput(25.75,13.5)(.0649671053,.0337171053){304}{\line(1,0){.0649671053}}
\put(56,33.375){\vector(-1,-1){.07}}\multiput(66,42.75)(-.035971223,-.0337230216){556}{\line(-1,0){.035971223}}
\put(56.25,29.625){\vector(-3,-2){.07}}\multiput(66.25,35.5)(-.0573065903,-.0336676218){349}{\line(-1,0){.0573065903}}
\put(55.875,19.125){\vector(-2,1){.07}}\multiput(65.75,14)(-.0649671053,.0337171053){304}{\line(-1,0){.0649671053}}
\put(55.75,14.125){\vector(-1,1){.07}}\multiput(66,4.25)(-.0349829352,.0337030717){586}{\line(-1,0){.0349829352}}
\put(5.5,20.25){$u$}
\put(25.5,45.25){$s_1$}
\put(25.5,35.75){$s_2$}
\put(25.5,10.5){$s_k$}
\put(45,20){$v$}
\put(68.25,43.25){$t_1$}
\put(68.5,35.5){$t_2$}
\put(68.75,14){$t_{k+1}$}
\put(68.5,3.5){$t_{k+2}$}
\end{picture}

Figure 2
\end{center}

\section{Oriented graphs}

Let $G=(V,E)$ be a graph. An orientation of $G$ is a digraph $D=(V,A)$ such that $uv\in E \Leftrightarrow (uv\in A$ or $vu\in A),$ and $|E|=|A|.$ Of course, two  distinct orientations of a given graph may have different majority domination numbers. This suggests the following definitions:

\begin{dfn}
Let $G$ be a graph. Then $dom^+_{maj}(G)=\min\{\gamma^+_{maj} (D): D \ \text{is an orientation of} \ G\}$ and $DOM^+_{maj}(G)=\max\{\gamma^+_{maj} (D): D \ \text{is an orientation of} \ G\}.$
  \end{dfn}

  The study of these two parameters is quite interesting. In contrast with what happens with majority out-dominating sets (see \cite{mase}), in this case it does not hold that for every graph $G,$ $dom^+_{maj}(G)=\gamma_{maj} (G).$

\begin{prop} \label{prop 334}
Let $P_n$ denote the (undirected) path with $n$ vertices. Then $dom^+_{maj}(P_n) =-n+2\lceil \frac{n+2}4 \rceil$ and
\[DOM^+_{maj}(P_n) = \left\lbrace
  \begin{array}{c l}
  0  & \text{if  $  n$ is even }\\
    1  & \text{if  $  n$ is odd }
  \end{array}
\right. \]
\end{prop}

\begin{proof}
For any orientation $D$ of $P_n$ and any vertex $v\in V(P_n)=V(D),$ $d^-_D(v)+d^+_D(v)\leq 2.$ Given any MODF $f$ of $D$, a vertex $v\in V(D)$ with $f(v)=-1$ will satisfy $f(N^+[v])\geq1$ if, and only if, $d^+(v)=2$ and $N^+(v)\subseteq f^{-1}(1).$  Therefore, if $|f^{-1}(1)|=k$ then  a maximum of $2k-1$ vertices $x$ satisfy $f(N^+[x])\geq1.$ Since $f$ is a MODF of $D,$ it follows that $2k-1\geq \lceil \frac{n}2 \rceil,$ which implies $k \geq \lceil \frac{n+2}4 \rceil,$ that is, $dom^+_{maj}(P_n) \geq -n+2\lceil \frac{n+2}4 \rceil.$

 On the other hand, consider the following orientation $D_1=(V,A)$ of $P_n$: We number the vertices of $V(P_n)$ in order, that is, $V(P_n)=\{v_1,...,v_n\},$ with $N(v_i)=\{v_{i-1},v_{i+1}\}$ for $i\in\{2,...,n-1\},$ $N(v_1)=\{v_2\},$ and $N(v_n)=\{v_{n-1}\}$; we orient the edges of $P_n$ in such a way that for a vertex $v_i\in V,$ $d^+(v_i)=0$ if, and only if, $i$ is even. Then the function $g:V\rightarrow\{-1,1\}$ such that $g(v_i)=1$ if $i$ is even and $i\leq \lceil \frac{n}2 \rceil+2,$ and $g(v_i)=-1$ otherwise, is a MODF of $D_1$ with $w(g)=-n+2\lceil \frac{n+2}4 \rceil.$

 For $DOM^+_{maj}(P_n),$ number the vertices of $V$ as in the previous paragraph and consider any orientation $D=(V,A)$  of $P_n:$

If $n$ is even and $v_{\frac{n}2}v_{\frac{n}2+1}\in A,$ the function $f:V\rightarrow\{-1,1\}$ such that $f(v_i)=1$ if  $i\geq  \frac{n}2 +1$ and $f(v_i)=-1$ otherwise is a MODF of $D$ with $w(f)=0.$

If $n$ is even and $v_{\frac{n}2+1}v_{\frac{n}2}\in A,$ the function $f:V\rightarrow\{-1,1\}$ such that $f(v_i)=1$ if  $i\leq \frac{n}2 $ and $f(v_i)=-1$ otherwise is a MODF of $D$ with $w(f)=0.$

If $n$ is odd and $d^-(v_{\lceil \frac{n}2 \rceil})=0$ or $d^-(v_{\lceil \frac{n}2 \rceil})=2,$ the function $f:V\rightarrow\{-1,1\}$ such that $f(v_i)=1$ if  $i\geq \lceil \frac{n}2 \rceil$ and $f(v_i)=-1$ otherwise is a MODF of $D$ with $w(f)=1.$

If $n$ is odd, $d^-(v_{\lceil \frac{n}2 \rceil})=1,$ and $v_{\lfloor \frac{n}2 \rfloor}v_{\lceil \frac{n}2 \rceil}\in A,$  the function $f:V\rightarrow\{-1,1\}$ such that $f(v_i)=1$ if  $i\geq \lceil \frac{n}2 \rceil$ and $f(v_i)=-1$ otherwise is a MODF of $D$ with $w(f)=1.$

If $n$ is odd, $d^-(v_{\lceil \frac{n}2 \rceil})=1,$ and $v_{\lceil \frac{n}2 \rceil}v_{\lfloor \frac{n}2 \rfloor}\in A,$  the function $f:V\rightarrow\{-1,1\}$ such that $f(v_i)=1$ if  $i\leq \lceil \frac{n}2 \rceil$ and $f(v_i)=-1$ otherwise is a MODF of $D$ with $w(f)=1.$

Therefore, $DOM^+_{maj}(P_n) \leq 0$ if $n$ is even, and $DOM^+_{maj}(P_n) \leq 1$ if $n$ is odd. Equality holds for directed paths, as shown in Proposition \ref{prop333}.
\end{proof}

\begin{prop} \label{prop 335}
Let $C_n$ denote the (undirected) cycle with $n$ vertices. Then  $dom^+_{maj}(C_n) =-n+2\lceil \frac{n+2}4 \rceil$ and
 \[DOM^+_{maj}(C_n) = \left\lbrace
  \begin{array}{c l}
  2  & \text{if  $  n$ is even }\\
    3  & \text{if  $  n$ is odd }
  \end{array}
\right. \]
\end{prop}

\begin{proof}
As in the proof of Proposition \ref{prop 334}, for any orientation $D$ of $P_n$ and any vertex $v\in V(P_n)=V(D),$ $d^-_D(v)+d^+_D(v)\leq 2.$ Moreover, given any MODF $f$ of $D$, a vertex $v\in V(D)$ with $f(v)=-1$ will satisfy $f(N^+[v])\geq1$ if, and only if, $d^+(v)=2$ and $N^+(v)\subseteq f^{-1}(1).$ It follows that for every MODF $f$ of $D$, $f^{-1}(1)\geq \lceil \frac{n+2}4 \rceil.$

On the other hand, it is easy to show that $dom^+_{maj}(C_3) = 1$ and $dom^+_{maj}(C_4) = 0.$ For $n\geq5,$ we observe that in the orientation $D_1$ and the function $g$ proposed in the proof of Proposition \ref{prop 334}, $\{v_1, v_n\}\subseteq g^{-1}(-1),$ which implies $g(N^+[v_1])\leq0$ and $g(N^+[v_n])\leq0.$ Therefore, $g$ is a MODF of $D_1+v_1v_n.$ This means that for any $\gamma^+_{maj}$-function $f$ of $D_1+v_1v_n,$ $f^{-1}(1)\leq \lceil \frac{n+2}4 \rceil.$ Since $D_1+v_1v_n$ is an orientation of $C_n,$ we have that  $dom^+_{maj}(C_n) =-n+2\lceil \frac{n+2}4 \rceil.$

For the value of $DOM^+_{maj}(C_n),$ we proceed in the following way: Number the vertices of $C_n$ in order, as in the proof of Proposition \ref{prop 334}, and consider any orientation $D$ of $C_n.$ Now take $D'=D-v_nv_1$ or $D'=D-v_1v_n,$ whichever applies, which is an orientation of the graph $C_n-v_nv_1,$ isomorphic to $P_n.$ According to Proposition \ref{prop 334}, there is a MODF $f$ of $D'$ with $|f^{-1}(1)|=\lceil \frac{n}2 \rceil,$ and such that either $f(v_1)=1$ or $f(v_n)=1.$  Then the function $g:V\rightarrow\{-1,1\}$ such that $g(v_1)=g(v_n)=1$ and $g(v_i)=f(v_i)$ for $i\in\{2,...,n-1\}$ is a MODF of $D,$ and $w(g)=\lceil \frac{n}2 \rceil+1.$ Therefore, $DOM^+_{maj}(C_n) \leq 2$ if $n$ is even, and $DOM^+_{maj}(C_n) \leq 3$ if $n$ is odd. Equality holds for directed cycles, as shown in Proposition \ref{prop34}.
\end{proof}

\begin{prop} \label{prop 336}
For the star $K_{1,n-1}$ we have:
\[DOM^+_{maj}(K_{1,n-1}) = \left\lbrace
  \begin{array}{c l}
  0  & \text{if  $  n$ is even }\\
    1  & \text{if  $  n$ is odd }
  \end{array}
\right. \]

Moreover, if $n\geq5$ then:
\[dom^+_{maj}(K_{1,n-1}) = \left\lbrace
  \begin{array}{c l}
  -2  & \text{if  $  n$ is even }\\
    -1  & \text{if  $  n$ is odd }
  \end{array}
\right. \]
\end{prop}

\begin{proof}
Take $K_{1,n-1}=(V,E)$ and let $v$ be its central vertex. For any orientation $D$ of $K_{1,n-1},$ any function $f:V\rightarrow \{-1,1\},$ and any vertex $u\in V\setminus \{v\},$ we have that $f(N^+[u])\geq1\Rightarrow f(u)=1.$ Therefore, for every MODF $g$ of $D,$ $|g^{-1}(1)|\geq\lceil\frac{n-2}2\rceil.$ This implies that $dom^+_{maj}(K_{1,n-1}) \geq -2$ if $n$ is even, and $dom^+_{maj}(K_{1,n-1}) \geq -1$ if $n$ is odd. On the other hand, let $n\geq5$ and  take an orientation $D_1$ of $K_{1,n-1}$ such that $d^+(v)=\lceil\frac{n-2}2\rceil.$ Then the function $h:V\rightarrow \{-1,1\}$ such that $h(u)=1$ if $u\in N^+(v)$ and $h(u)=-1$ otherwise is a MODF of $D_1,$ so $dom^+_{maj}(K_{1,n-1}) = -2$ if $n$ is even, and $dom^+_{maj}(K_{1,n-1}) = -1$ if $n$ is odd.

Now take any orientation $D$ of $K_{1,n-1}.$ If $d^+(v)\geq \lceil\frac{n-2}2\rceil,$ the function $f_1:V\rightarrow \{-1,1\}$ such that $f_1^{-1}(1)= \{v\}\cup S_1$ where $S_1\subseteq N^+(v),$ $|S_1|=\lceil\frac{n-2}2\rceil,$ is a MODF of $D.$  If $d^+(v)< \lceil\frac{n-2}2\rceil,$ the function $f_2:V\rightarrow \{-1,1\}$ such that $f_2^{-1}(1)= N^+[v]\cup S_2$ where $S_2\subseteq N^-(v),$ $|S_2|=\lceil\frac{n-2}2\rceil-|N^+(v)|,$ is a MODF of $D.$ Therefore, $DOM^+_{maj}(K_{1,n-1}) \leq 0$ if $n$ is even, and $DOM^+_{maj}(K_{1,n-1}) \leq 1$ if $n$ is odd. Moreover, if $D'$ is the orientation of $K_{1,n-1}$ such that $d^+(v)=0,$ then $\gamma^+_{maj}(D')= 0$ if $n$ is even, and $\gamma^+_{maj}(D')= 1$ if $n$ is odd. This completes the proof.
\end{proof}

\begin{dfn}
A double star is a graph resulting from joining the central vertices of two disjoint stars.
\end{dfn}

\begin{prop}
Let $G=(V,E)$ be a double star. Then:
 \[DOM^+_{maj}(G) = \left\lbrace
  \begin{array}{c l}
  0  & \text{if  $  n$ is even }\\
    1  & \text{if  $  n$ is odd }
  \end{array}
\right. \]

Moreover, if $n\geq13$ and for every $v\in V$ we have $d(v)\neq2,$ then:
\[dom^+_{maj}(G) = \left\lbrace
  \begin{array}{c l}
  -4  & \text{if  $  n$ is even }\\
    -3  & \text{if  $  n$ is odd }
  \end{array}
\right. \]

and if $5\leq n<13,$ or $n\geq13$ and there exists $u\in V$ with $d(u)=2,$ then:
\[dom^+_{maj}(G) = \left\lbrace
  \begin{array}{c l}
  -2  & \text{if  $  n$ is even }\\
    -1  & \text{if  $  n$ is odd }
  \end{array}
\right. \]
\end{prop}

\begin{proof}
Let $u$ and $v$ be the stem vertices of $G,$ and let $D=(V,A)$ be any orientation of $G$ with $uv\in A.$ We define a function $f:V\rightarrow \{-1,1\}$ in the following way:

 If $d^+(u)+d^+(v)-1\geq \lceil\frac{n}2\rceil,$ we take a set $S_1\subseteq (N^+_D(u)\cup N^+_D(v))\setminus\{v\}$ with $|S_1|=\lceil\frac{n}2\rceil,$  and assign $f(x)=1$ for $x\in S_1,$ $f(x)=-1$ for $x\in V\setminus S_1.$

 If $d^+(u)+d^+(v)= \lceil\frac{n}2\rceil,$ we assign $f(x)=1$ for $x\in N^+_D(u)\cup N^+_D(v),$ $f(x)=-1$ otherwise.

 If $d^+(u)+d^+(v)= \lceil\frac{n}2\rceil-1,$ we assign $f(x)=1$ for $x\in N^+_D[u]\cup N^+_D(v),$ $f(x)=-1$ otherwise.

 If $d^+(u)+d^+(v)= \lceil\frac{n}2\rceil-k,$ for $2\leq k\leq \lceil\frac{n}2\rceil-1,$ we take a set $S_2\subseteq N^-_D(u)\cup (N^-_D(v)\setminus\{u\})$ with $|S_2|= k-1,$  and assign $f(x)=1$ for $x\in N^+_D[u]\cup N^+_D(v)\cup S_2,$ $f(x)=-1$ otherwise.

 It is clear that in each case the function $f$ is a MODF of $D$ with $|f^{-1}(1)|=\lceil\frac{n}2\rceil.$ Therefore, $DOM^+_{maj}(G) \leq 0$ if $n$ is even, and $DOM^+_{maj}(G) \leq 1$ if $n$ is odd.

 On the other hand, if $D'$ is the orientation of $G$ such that $d^+(u)=1$ and $d^+(v)=0,$ it is easy to see that $\gamma^+_{maj}(D')= 0$ if $n$ is even, and $\gamma^+_{maj}(D')= 1$ if $n$ is odd.

Now we will prove the statements for $dom^+_{maj}(G)$: The cases with $n<13$ can be  easily verified, so we assume $n\geq 13.$ As in the proof of Proposition \ref{prop 336}, for any orientation $D$ of $G,$ any function $f:V\rightarrow \{-1,1\},$ and any pendant vertex $x,$ we have that $f(N^+[x])\geq1\Rightarrow f(x)=1.$ Therefore, for every MODF $f$ of $D,$ $|f^{-1}(1)|\geq\lceil\frac{n-4}2\rceil,$ that is, $dom^+_{maj}(G)\geq -4$ if $n$ is even and $dom^+_{maj}(G)\geq -3$ if $n$ is odd. Let $u$ and $v$ be the stem vertices of $G,$ and without loss of generality assume $uv\in A.$ For equality to hold we need both $f(N^+[u])$ and $f(N^+[v])$ positive with $f(u)=f(v)=-1.$ This is possible only if $u$ has at least three pendant vertices with value $1$ in its out-neighborhood, and $v$  has at least two, which implies that $G$ has no vertex of degree $2.$ In this case, for the orientation $D'$ of $G$ such that  $|N^+(u)|\geq4,$ $|N^+(v)|\geq2,$ and $|N^+(u)\cup N^+(v)|=\lceil\frac{n-4}2\rceil+1,$ we define the function $g:V\rightarrow \{-1,1\},$ such that $g(x)=1$ if $x\in (N^+(u)\cup N^+(v))\setminus \{v\},$ and $g(x)=-1$ otherwise. Then $\gamma^+_{maj}(D')=-4$ if $n$ is even, and $\gamma^+_{maj}(D')=-3$ if $n$ is odd.

On the other hand, suppose $G$ has a vertex of degree $2.$ As stated in last paragraph, in such a case it is not possible to have $f(u)=f(v)=-1$ with $f(N^+[u])$ and $f(N^+[v])$ positive. However, it can be done for one of the stem vertices, namely $v$. Therefore, for any function $f:V\rightarrow \{-1,1\},$ we have $|f^{-1}(1)|\geq\lceil\frac{n-2}2\rceil,$ that is, $dom^+_{maj}(G)\geq -2$ if $n$ is even, and $dom^+_{maj}(G)\geq -1$ if $n$ is odd. Now consider an orientation $D''$ of $G$ such that $|N^+(v)|=\lceil\frac{n-2}2\rceil,$ and define the function $g:V\rightarrow \{-1,1\},$ such that $g(x)=1$ if $x\in N^+(v),$ and $g(x)=-1$ otherwise. Since $w(g)=-2$ if $n$ is even, and $w(g)=-1$ if $n$ is odd, the proof is complete.
\end{proof}

\begin{prop} \label{prop353}
Let $G=(V,E)$ be a double star. Then $\gamma_{maj}(G)=DOM^+_{maj}(G)$.
\end{prop}

\begin{proof}
Let $G=(V,E)$ be a graph and let $f:V\rightarrow \{-1,1\}$ be a $\gamma_{maj}(G)$-function. It is easy to verify that for $n=4$, $\gamma_{maj}(G)=0$ (the result appears as well in \cite{mf}). Assume $n\geq5$ and let $u$ and $v$ be the stem vertices of $G.$ If $f(u)=f(v)=-1,$ then $f(N[x]\leq0$ for every $x\in V\setminus \{u,v\},$ so $f$ is not a majority dominating function.

Suppose $d(u)=-1$; if there is $z\in N(u)\setminus \{v\}$ with $f(z)=1,$ then the function $g:V\rightarrow \{-1,1\}$ such that $g(u)=1,$ $g(z)=-1,$ and $g(x)=f(x)$ for every $x\in V\setminus \{u,z\}$ is a $\gamma_{maj}(G)$-function with $g(u)=g(v)=1,$ since $g(N[x])\geq f(N[x])$ for $x\in N[u],$ and $g(N[x])= f(N[x])$ for $x\notin N[u].$ Therefore, we can assume that if $f(u)=-1$ then $f(x)=-1$ for every $x\in N[u]\setminus \{v\},$ which implies $f(N[x])\leq 0$ for every $x\in N[u]\setminus \{v\}.$ Now, for every $x\in N(v)\setminus \{u\}$ we have $f(N[x])\geq 0$ if, and only if, $f(x)=1,$ so for every $x\in V$ we have that $f(x)=-1 \Rightarrow f(N[x])\leq 0.$ Therefore, $|f^{-1}(1)|\geq \lceil\frac{n}2\rceil,$ that is, $w(f)\geq0$ if $n$ is even and $w(f)\geq1$ if $n$ is odd.

  Now suppose $f(u)=f(v)=1.$ Then for every $x\in V\setminus \{u,v\}$ we have $f(N[x])\geq 0$ if, and only if, $f(x)=1.$ Therefore,  $|f^{-1}(1)|\geq \lceil\frac{n-4}2\rceil+2,$ that is,  $w(f)\geq0$ if $n$ is even and $w(f)\geq1$ if $n$ is odd. So in every case we have $\gamma_{maj}(G)\geq DOM^+_{maj}(G)$.

  On the other hand, without loss of generality assume $d(u)\leq d(v),$ and consider the function $h:V\rightarrow \{-1,1\}$ such that $h(x)=1$ for $x\in \{v\}\cup S,$ where $S\subseteq N[v]\setminus \{u\}$ and $|S|= \lceil\frac{n-2}2\rceil,$ and $h(x)=-1$ otherwise. It is clear that $h$ is a majority dominating function of $G$ with $w(h)=0$ if $n$ is even and $w(h)=1$ if $n$ is odd.
\end{proof}

\begin{prop} \label{prop 355}
For any two positive integers $2\leq r\leq s,$  $dom_{maj}^+ (K_{r,s})=4-n.$
\end{prop}

\begin {proof}
Take $K_{r,s}=(V,E)$ as in the hypothesis, and consider a vertex $v\in V.$ Let $D$ be an orientation of $K_{r,s}$, and let $g:V\rightarrow \{-1,1\}$ be a function such that $g(v)=1$ and $g(x)=-1$ for $x\in V\setminus \{v\}.$ Then $g(N^+[x])\leq0$ for every $x\in V\setminus \{v\},$ since $g(x)=-1$ and $|g^{-1}(1)|=1.$ Therefore, $g$ is not a MODF of $D$, because $r+s\geq4.$ This implies that for every MODF $f$ of $D$ we have $|f^{-1}(1)|\geq2,$ that is,  $dom^+_{maj}(K_{r,s})\geq4-n.$

On the other hand, let $R$ and $S$ be the defining partite sets of $K_{r,s},$ with $|R|=r$ and $|S|=s,$ and take $\{u,v\}\subseteq R.$ Take the orientation $D'$ of $K_{r,s}$ such that $d^+(u)=d^+(v)=0$ and $d^+(x)=s$ for every $x\in R\setminus \{u,v\}.$ Consider the function $h:V\rightarrow \{-1,1\}$ such that $h(u)=h(v)=1$ and $h(x)=-1$ otherwise. Since $|h^{-1}(1)|=2,$ then $w(h)=4-n.$ Moreover, observe that $h(N^+[y])=1$ for every $y\in S,$ so $h$ is a MODF of $D'.$
\end{proof}

In relation with Proposition \ref{prop 355}, we have the following conjecture:

\begin{con}
Let $2\leq r\leq s$ be two integers. Then:
 \[DOM^+_{maj}(K_{r,s}) = \left\lbrace
  \begin{array}{c l}
  2  & \text{if  $  r+s$ is even }\\
    3  & \text{if  $  r+s$ is odd }
  \end{array}
\right. \]
\end{con}

\begin{theorem} \label{prop 356}
For any graph $G$, we have $dom^+_{maj}(G)\leq \gamma_{maj}(G).$
\end{theorem}

\begin{proof}
Let $G=(V,E)$ be a graph and let $f:V\rightarrow \{-1,1\}$ be a $\gamma_{maj}(G)$-function. We get the orientation $D=(V,A)$ of $G$ as follows: For any $uv\in E$ with $f(u)=-1$ and $f(v)=1,$ the arc $uv\in A.$ Edges whose vertices are both positive or both negative may be oriented arbitrarily. We will see that $f$ is a MODF of $D$: Take $u\in V.$ If $f(u)=-1,$ then $f^{-1}(1)\cap N_G(u)\subseteq N^+_D(u),$ and $f^{-1}(-1)\cap N^+_D(u)\subseteq N_G(u),$ so $f(N_G[u])\leq f(N^+_D[u]).$ If $f(u)=1,$ then $f^{-1}(-1)\cap N^+_D[u]=\emptyset.$ Therefore, $f$ is a MODF of $D$, so $dom^+_{maj}(G)\leq w(f)= \gamma_{maj}(G).$
\end{proof}

In relation with Theorem \ref{prop 356}, given a graph $G$ we may have $\gamma_{maj}(G)<DOM^+_{maj}(G)$, $\gamma_{maj}(G)=DOM^+_{maj}(G)$, or $\gamma_{maj}(G)>DOM^+_{maj}(G)$. In \cite{mf} it is proven that for paths ($n\geq2$) and cycles ($n\geq3$), $\gamma_{maj}(G)=-2\lceil\frac{n-4}6\rceil$ for $n$ even, and $\gamma_{maj}(G)=1-2\lceil\frac{n-3}6\rceil$ for $n$ odd, so Propositions \ref{prop 334} and \ref{prop 335} imply that for paths and cycles $\gamma_{maj}(G)<DOM^+_{maj}(G)$. It appears as well in \cite{mf} that $\gamma_{maj}(K_{1,n-1})=1$ if $n$ is even, and $\gamma_{maj}(K_{1,n-1})=2$ if $n$ is odd, which along with Proposition \ref{prop 336} means that  $DOM^+_{maj}(K_{1,n-1})<\gamma_{maj}(K_{1,n-1}).$ Proposition \ref{prop353} shows that for double stars $\gamma_{maj}(G)=DOM^+_{maj}(G)$.

\section{Complexity}

In this section we will prove that the decision problem MAJORITY OUT-DOMINATING FUNCTION is NP-complete. This will be accomplished by means of a polynomial reduction from a particular case of the problem IN-DOMINATING SET, known to be NP-complete.  The statements of the problems mentioned above are as follows: \\

MAJORITY OUT-DOMINATING FUNCTION (MODF)

\emph{Instance}: A digraph $D$ and an integer $k\leq n.$

\emph{Question}: Is there a majority out-dominating function of $D$ with weight $k$ or less? \\

IN-DOMINATING SET

\emph{Instance}: A digraph $D'$ of order $n$ with regular out-degree $d>\frac{n-2}4,$ and a positive integer $k< \frac{n}2+1.$

\emph{Question}: Is there an in-dominating set of $D'$ with cardinality $k$ or less? \\

It was proven in \cite {lee} (result appearing as well in \cite {h}) that  $\gamma^-\leq(\frac{\delta^++1}{2\delta^++1})n$. Therefore, our choice of $d$ guarantees $\gamma^-<\frac{n}{2}+1$. As a comment, the quoted result is stated for $\delta^+\geq1$ but holds as well if $\delta^+=0.$

\begin{theorem}
The decision problem MAJORITY OUT-DOMINATING FUNCTION (MODF) is NP-complete.
\end{theorem}

\begin{proof}
It is clear that MODF is in NP.

Consider a  digraph $D'$ of order $n$ with regular out-degree $d>\frac{n-2}4$ and a positive integer $k< \frac{n}2+1.$ Take a complete digraph $T$ of order $n+2d$ (that is, for every $\{u,v\}\subseteq V(T),$ both $uv$ and $vu$ are in $A(T)$), and an empty (di)graph $D''$ of order $2d.$ Let $X$ be a subset of $V(T)$ with $|X|=d.$ We construct the digraph $D$ from the disjoint union of $T, \ D',$ and $D''$ by adding symmetric arcs between every vertex of $D''$ and every vertex of $X,$ and symmetric arcs between every vertex of $D'$ and every vertex of $X.$ It is clear that $D$ can be constructed in polynomial time.

Let $S$ be an in-dominating set of $D'$ of cardinality at most $k.$ We will show that there is a majority out-dominating function of $D$ of weight at most $2k-2n-2d$: Consider the function $f:V(D)\rightarrow\{-1,1\}$ such that $f(v)=1 $ if, and only if, $v\in S\cup X.$ Then the weight of $f$ is at most $2|S|-2n-2d\leq2k-2n-2d.$ If $v\in S,$ then $f(v)=1,$  $v$ has $d$ out-neighbors  in $V(D')$ and $d$ out-neighbors in $X,$ so $f(N^+[v])\geq 1.$ If $v\in V(D')\setminus S,$ then $v$ has $d$ out-neighbors in $X,$ at least one out-neighbor in $S,$  and at most $d-1$ out-neighbors in $V(D')\setminus S,$ which implies that $f(N^+[v])\geq 1.$ Moreover, for every vertex $v\in V(D''),$ $f(N^+[v])= d-1.$ Therefore, $f(N^+[v])\geq 1$ for at least $n+2d$ vertices in $V(D).$ Since the order of $D$ is $2n+4d,$ it follows that  $f$ is a majority out-dominating function in $D.$

Conversely, assume that $\gamma^+_{maj}(D)\leq 2k-2n-2d,$ and let $f$ be a $\gamma^+_{maj}(D)$-function such that $|X\cap f^{-1}(1)|$ is maximum. First, notice that $|f^{-1}(1)|\leq k+d,$ since otherwise $\gamma^+_{maj}(D)=|f^{-1}(1)|-|f^{-1}(-1)|\geq k+d+1-2n-3d+k+1=2k-2n-2d+2.$

Second, we will show that $f(N^+[v])\leq0$ for every $v\in V(T)$: Suppose that there exists $u\in V(T)$ such that $f(N^+[u])>0.$ If $u\in X,$ since $N^+[u]=V(D),$ we have $1\leq f(V(D))=\gamma^+_{maj}(D)\leq 2k-2n-2d,$ which implies that $n+d< k,$ a contradiction to the hypothesis $k< \frac{n}2+1.$ If $u\notin X,$ then $N^+[u]=V(T),$ so $f(v)=1$ for at least $\frac{n}2 +d+1$ vertices. Since  $k< \frac{n}2+1,$ we have $|f^{-1}(1)|> k+d,$ which is a contradiction. Therefore, $f(N^+[v])\leq0$ for every $v\in V(T).$ This implies that $f(N^+[v])>0$ for every $v\in V(D')\cup V(D'').$

Now we will show that $f(v)=1$ for every $v\in X$: Suppose that there exists $u\in X$ such that $f(u)=-1.$ If $f(z)=-1$ for every $z\in V(D'),$ then $f(N^+[z])\leq -3$ for every $z\in V(D'),$ which is a contradiction. So there exists $x\in V(D')$ such that $f(x)=1.$ Consider the function $g:V(D)\rightarrow \{-1,1\}$ such that $g(u)=1, \ g(x)=-1,$ and $g(y)=f(y)$ for every $y\in V(D)\setminus \{u,x\}.$ Since $N^-[u]=V(D),$ we have that if $y\in N^-[x],$ then $g(N^+[y])=f(N^+[y]),$ and if $y\notin N^-[x],$ then $g(N^+[y])=f(N^+[y])+2.$ Since $w(g)=w(f),$ $g$ is a $\gamma^+_{maj}(D)$-function such that  $|X\cap g^{-1}(1)|>|X\cap f^{-1}(1)|,$ contradicting the fact that $|X\cap f^{-1}(1)|$ is maximum among the $\gamma^+_{maj}(D)$-functions. Therefore, $f(v)=1$ for every $v\in X.$

Now take $S=V(D')\cap f^{-1}(1)$ and consider a vertex $v\in V(D').$ Since $f(N^+[v])\geq1,$ then there exists $u\in N^+_{D'}[v]$ such that $f(u)=1.$ This implies that either $v\in S$ or $v$ is in-dominated by a vertex in $S,$ that is, $S$ is an in-dominating set of $D'.$ Since $X\cap S=\emptyset,$ $X\cup S\subseteq f^{-1}(1),$ and $|f^{-1}(1)|\leq k+d,$ it follows that $|S|\leq k.$
\end{proof}

\section{Conclusions and scope}
In this paper we extended the notion of majority dominating function to digraphs. In addition to its applications, the topic is of mathematical interest since the behavior of MODFs is quite different to that of their counterparts in graphs. This is only an introductory work, in which the concept is defined and some basic results are proven.

Basically, two directions of research are suggested through the text, other than getting bounds or actual values for $\gamma^+_{maj}$ ($dom^+_{maj}$ and $DOM^+_{maj}$) of specific classes of digraphs (graphs). One of them is the proof of Conjecture \ref{con 36}, at least for some classes of digraphs; this could help to obtain results on complexity. The other direction is to explore the relation of $\gamma_{maj}$ with $dom^+_{maj}$ and $DOM^+_{maj}$; for example, characterize graphs $G$ satisfying  $\gamma_{maj}(G)<DOM^+_{maj}(G)$, $\gamma_{maj}(G)=DOM^+_{maj}(G)$, and $\gamma_{maj}(G)>DOM^+_{maj}(G)$, as well as those  for which  $\gamma_{maj}(G)=dom^+_{maj}(G)$, like $C_3$.

We hope this paper will be helpful for people working in related topics, and perhaps it will encourage further research in the field.

\end{document}